\documentclass[times]{article}

\usepackage{graphicx}
\usepackage{latexsym}
\usepackage{amsmath}
\usepackage{amssymb}
\usepackage{amsfonts}
\usepackage{verbatim}
\usepackage{mathrsfs}
\usepackage{color}
\usepackage[colorlinks,citecolor=blue,urlcolor=blue]{hyperref}

\newtheorem{tm}{Theorem}[section]
\newtheorem{rk}{Remark}[section]
\newtheorem{ap}{Assumption}[section]

\newcommand{\ee}{\mathbb E}
\newcommand{\pp}{\mathbb P}
\newcommand{\nn}{\mathbb N}
\newcommand{\rr}{\mathbb R}
\newcommand{\hh}{\mathbb H}

\newcommand{\bs}{\mathbf s}

\newcommand{\LL}{\mathcal L}
\newcommand{\OO}{\mathcal O}
\newcommand{\RR}{\mathcal R}

\newcommand{\OOO}{\mathscr O}
\newcommand{\FFF}{\mathscr F}

\begin{document}

\title{Well-posedness and Finite Element Approximations for Elliptic SPDEs with Gaussian Noises}

\author{
   Yanzhao Cao \thanks{Department of Mathematics, Auburn University, Auburn, AL 36849, USA
      ({\it yzc0009@auburn.edu}).}
     \and Jialin Hong \thanks{Academy of Mathematics and Systems Science, Chinese Academy of Sciences, Beijing, China
     ({\it hjl@lsec.cc.ac.cn}).} 
     \and Zhihui Liu \thanks{ Department of Mathematics, 
The Hong Kong University of Science and Technology, 
Hong Kong ({\it zhliu@ust.hk}). }
  }
     \date{}
\maketitle

%
 
\date{}

%
%
%
%
%
%
%
%
%
%

%
 
\date{}
%
%
%
%

\begin{abstract}

The paper studies the well-posedness and optimal error estimates of spectral finite element approximations for the boundary value problems of semi-linear elliptic SPDEs driven by white or colored Gaussian noises. The noise term is approximated through the spectral projection of the covariance operator, which is not required to be commutative with the Laplacian operator.  Through the convergence analysis of SPDEs with the noise terms replaced by the projected noises, the well-posedness of the SPDE is established under certain covariance operator-dependent conditions. These SPDEs with projected noises are then numerically approximated with the finite element method. A general error estimate framework is established for the finite element approximations. 
Based on this framework, optimal error estimates of finite element approximations for elliptic SPDEs driven by power-law noises are obtained. It is shown that with the proposed approach, convergence order of white noise driven SPDEs is improved by half for one-dimensional problems, and by an infinitesimal factor for higher-dimensional problems.
\end{abstract}

\maketitle




\section{Introduction}

In recent years, random disturbance as a form of uncertainty has been increasingly treated as an essential modeling factor in the analysis of complex phenomena. 
Adding such uncertainty to partial differential equations (PDEs) which model such physical and engineering phenomena, one derives stochastic PDEs (SPDEs) as improved mathematical modeling tools. 
SPDEs derived from fluid flows and other engineering fields are often assumed to be driven by white noises that have constant power spectral densities \cite{Gri95}. However, most random fluctuations in complex systems are correlated acting on different frequencies in which case the noises are called colored noises \cite{KMS88}. 

Elliptic SPDEs driven by white noises and colored noises have been considered by many authors, see e.g. \cite{ANZ98, CYY07, DZ02, Zhang2017, Zhang2016} for white noises, \cite{MS06, MM05, Zhang2017} for colored noises determined by Riesz-type kernels, \cite{CHL17, CHL18} for fractional noises, and \cite{WGBS15} for power-law noises. 

The main objective of this study is to investigate the well-posedness and optimal error estimate of spectral finite element approximations for the following semilinear elliptic SPDE:
\begin{align}\label{ell}
\begin{split}
-\Delta u(x, \xi )&=f(u(x, \xi ))+\dot{W}^Q (x, \xi ),\quad x\in \OOO, \ \ \xi \in \Omega, \\
u(x, \xi )&=0,\qquad\qquad\qquad\quad\ x\in \partial \OOO, \ \ \xi \in \Omega. 
\end{split}
\end{align}
Here $(\Omega,\FFF,\pp)$ is a probability space and $\OOO\subset \rr^d$ is a bounded domain with regular boundary $\partial \OOO$, $f:\rr \rightarrow \rr$ is a Lipschitz continuous function, and $\dot{W}^Q$ is a class of centered Gaussian noises with covariance operator $Q$. 

The existence of the unique solution for SPDE \eqref{ell} driven by the white noise, i.e., $Q=I$, has been established in \cite{BP90} by converting the problem into an integral equation. 
In this paper, we establish a covariance operator-dependent condition for the well-posedness of SPDE \eqref{ell} through the convergence analysis for a sequence of solutions of SPDEs with the noise term in SPDE \eqref{ell} replaced by its spectral projections. This sequence of SPDEs will play an important role in constructing our numerical solutions for SPDE \eqref{ell}. 

To obtain numerical solutions, we apply the finite element method to the aforementioned SPDEs whose noises are the spectral projections of the original noise. In previous studies \cite{ANZ98, DZ02}, the noises are approximated by piecewise constants in space. In addition, it is required that the eigenvectors of the Laplacian also diagonalize the covariance operator of the noise, i.e., the Laplacian operator and the covariance operators are commutative. In this study, the commutative assumption is no longer required. Another improvement over the results of \cite{ANZ98, DZ02} is that our numerical solutions achieve better convergence rates. In particular, for SPDEs driven by white noises, we obtain 1.5 order convergence instead of first order in the one-dimensional case,  and we improve the convergence order by an infinitesimal factor in the higher-dimensional cases. 

The paper is organized as follows. At the end of this section, we provide several notations that will be used throughout the rest of the paper. In Section \ref{sec-spe} we derive sufficient and necessary conditions for the existence of the unique solution for SPDE \eqref{ell} and establish its regularity. The error estimate between the solution of the SPDE driven by  the spectral truncation of the original noise and the exact solution of SPDE \eqref{ell}  is also derived. In Section \ref{sec-fem}, we construct finite element approximations to the SPDEs driven by the spectral truncation noises and derive the error estimate. As an application, we use this error estimate to derive optimal order of convergence of the finite element approximations for SPDEs driven by power-law noises and white noises.\\

For $r\in \nn$, we use $(\hh^r,\|\cdot\|_r)$ to denote the usual Sobolev space
\begin{align*}
\hh^r:=\left\{v:\ \|v\|_{\hh^r}:=\left(\sum_{|k|\le r} \|D^k v\|^2\right)^{1/2} <\infty\right\}. 
\end{align*}
When $r=0$, $\hh^0:=\hh$ is the space of square integrable functions on $\OOO$, whose inner product and norm are denoted by $(\cdot,\cdot)$ and $\|\cdot\|$, respectively. 
We also use $\hh^1_0$ (resp., $\hh_0$) to denote the subspace of $\hh^1$ (resp., $\hh$) whose elements vanish on $\partial \OOO$. 
For $s\in \rr$, we use $(\dot\hh^s,\|\cdot\|_s)$ to denote the interpolation space
\begin{align*}
\dot\hh^s:=\left\{v:\ \|v\|_s
:=\left(\sum_{k\in \nn_+} \lambda_k^s(v,\varphi_k)^2\right)^{1/2} <\infty\right\},
\end{align*}
where $\{(\lambda_k,\varphi_k)\}_{k\in \nn_+}$ is an eigensystem of the negative Dirichlet Laplacian.
When $\bs\in \nn$, it is known (see e.g. \cite{Tho06}, Lemma 3.1) that $\dot\hh^\bs$ coincides with the usual Sobolev space $\hh^\bs$ with additional boundary conditions.
In the rest of the paper, we will denote by $C$ a generic positive constant which is independent of the number of spectral truncation terms and the mesh size of finite element triangulations and may differ from one place to another.

\section{Spectral Approximations}
\label{sec-spe}

In this section, we prove the existence of a unique solution of SPDE \eqref{ell} through the spectral projection of the noise, and establish its Sobolev regularity.
We also derive the error estimate between the solution of the SPDE driven by the spectral truncations and the exact solution of \eqref{ell}.

\subsection{Formulations}

Recall that an $\hh$-valued random field $u=\{u(x):x\in \OOO\}$ is said to be a mild solution of SPDE \eqref{ell} if 
\begin{align}\label{mild}
u=A^{-1}f(u)+A^{-1}\dot{W}^Q,\quad \text{a.s.}
\end{align}
Here $A^{-1}=(-\Delta)^{-1}$ is the inverse of negative Dirichlet Laplacian.

For the bounded and open domain $\OOO$ with piecewise smooth boundary $\partial \OOO$, $A$ has discrete and nonnegative eigenvalues $\{\lambda_k\}_{k=1}^\infty$ in an ascending order with finite multiplicity and corresponding smooth eigenvectors $\{\varphi_k\}_{k=1}^\infty$, which vanish on $\partial \OOO$ and form a complete orthonormal basis in $\hh_0$ (see e.g. \cite{GN13}), i.e.,
\begin{align}\label{a}
A\varphi_k=\lambda_k\varphi_k,\quad k\in \nn_+. 
\end{align}
Moreover, the asymptoticity of these eigenvalues is characterized by Weyl's law (see e.g. \cite{GN13}):
\begin{align}\label{weyl}
\lambda_k\asymp k^{\frac 2d},\quad \text{as}\quad k\rightarrow \infty,
\end{align}
where the notation $A\asymp B$ means that there exists a generic positive constant $C$ such that $C^{-1}A\le B\le CA$.
The Weyl's law is our main tool in the optimal error estimation of finite element approximations for SPDE \eqref{ell} with power-law noises.

The centered Gaussian noise $\dot{W}^Q$ is uniquely determined by its covariance operator $Q$.
Assume that $Q$ has $\{(\sigma_k,\psi_k)\}_{k=1}^\infty$ as its eigensystem, i.e.,
\begin{align}\label{q}
Q\psi_m=\sigma_m\psi_m,\quad m\in \nn_+, 
\end{align}
where $\{\psi_k\}_{k=1}^\infty$ form a complete orthonormal basis in $\hh$.
Then we have the following expansion for the infinite dimensional noise $\dot{W}^Q$:
\begin{align}\label{k-l}
\dot{W}^Q(\omega)=\sum_{m=1}^\infty Q^{\frac 12}\psi_m \eta_m(\omega),\quad \omega\in \Omega,
\end{align}
where $\{\eta_m\}_{m=1}^\infty$ are independent and normal random variables. 

To ensure the well-posedness of SPDE \eqref{ell}, we make the following assumption on $f$.

\begin{ap}\label{f}
$f$ is Lipschitz continuous, i.e., 
\begin{align}\label{lip}
\|f\|_{\text{Lip}}:=
\sup_{u,v\in \rr, u\neq v} \frac{|f(u)-f(v)|}{|u-v|}<\infty, 
\end{align}
and the Lipschitz constant $\|f\|_{\text{Lip}}$ is smaller than the positive constant $\gamma$ in the Poincar\'{e} inequality:
\begin{align}\label{poin}
\|\nabla v\|^2\geq \gamma\|v\|^2,\quad \forall\ v\in \hh^1_0. 
\end{align}
\end{ap}


The key assumption for the well-posedness of the SPDE \eqref{ell} is following boundness on the partial differential operator $A$ with respect to the Hilbert-Schmidt norm relating to the covariance operator $Q$. 
\begin{ap}\label{w}
There exists a constant $\beta\in [0,2]$ such that
\begin{align}\label{wq}
\|A^{\frac {\beta-2}2}\|_{\LL_2^0}<\infty,
\end{align}
where $\LL_2^0:=HS(Q^\frac 12(\hh), \hh)$ denotes the space of Hilbert-Schmidt operators from $Q^\frac 12(\hh)$ to $\hh$ and $\|\cdot \|_{\LL_2^0}$ denotes the corresponding norm.
\end{ap}
\begin{rk}
$Q$ is a trace operator if and only if \eqref{wq} holds for $\beta=2$. Another class of noises satisfying Assumption \ref{w} are the power-law noises, where
$Q=A^\rho$ for certain $\rho\in \rr$, or equivalently, 
 $\psi_k=\varphi_k$ and $\sigma_k=\lambda_k^\rho$ for all 
 $k\in \nn$.
In particular, if $\rho=0$, the power-law noise becomes the white noise \cite{DZ02}. As pointed out in \cite{WGBS15}, power-law noises abound in nature and have been observed extensively in both time series and spatially varying environmental parameters.
\end{rk}

\subsection{Well-posedness and Regularity}

The parameter $\beta$ in \eqref{wq} determines the regularity of $\dot{W}^Q$.
In fact, 
\begin{align}\label{w-reg}
\ee \left[\|\dot{W}^Q\|_{\beta-2}^2\right]
&=\ee \left[\left\|\sum_{k=1}^\infty A^{\frac {\beta-2}2} Q^{\frac 12}\psi_k \eta_k \right\|^2\right] \nonumber \\
&=\ee \left[\sum_{m=1}^\infty \left(\sum_{k=1}^\infty 
\left(A^{\frac {\beta-2}2} Q^{\frac 12}\psi_k,\varphi_m \right) \eta_k\right)^2 \right] \nonumber  \\
&=\sum_{m=1}^\infty \sum_{k=1}^\infty 
\left(A^{\frac {\beta-2}2} Q^{\frac 12}\psi_k,\varphi_m \right)^2 \nonumber \\
& =\left\|A^{\frac {\beta-2}2}\right\|_{\LL_2^0}^2
<\infty,
\end{align}
which shows that $\dot{W}^Q\in L^2(\Omega; \dot \hh^{\beta-2})$.

%

Now we consider the well-posedness and regularity of SPDE \eqref{ell}. Let $\pp_N: \hh\rightarrow V_N=\text{span}\{\varphi_m\}_{m=1}^N$ be the projection operator from $\hh$ to $V_N$: 
$(\pp_N u,v)=(u,v)$ for any $u\in \hh$, $v\in V_N$, and let $u_N$ be the solution of the SPDE with the noise term in \eqref{mild} replaced by its spectral projection: 
\begin{align}\label{eq:app_SPDE}
A u_N=f(u_N)+\pp_N \dot{W}^Q,\quad N\in \nn_+. 
\end{align}
Then
\begin{align}\label{spe-mild}
u_N=A^{-1} f(u_N)+A^{-1} \pp_N \dot{W}^Q,\quad N\in \nn_+.
\end{align}

\begin{tm}\label{well}
Let $p\ge 1$ and Assumptions \ref{f} and \ref{w} hold. 
Then SPDE \eqref{ell} possesses a unique solution $u\in L^p(\Omega; \dot \hh^\beta)$.
\end{tm}

\textbf{Proof:}  
We first prove the existence of an $L^p(\Omega; \hh)$-valued solution. 
For each $N\in \nn_+$, the existence of a unique solution $u_N\in \hh^1_0$ a.s. of Eq. \eqref{spe-mild} follows from the classical elliptic PDE theory. 
For $M<N$, set 
$E_{M,N}:=A^{-1}(\pp_N \dot{W}^Q-\pp_M \dot{W}^Q)$. Then
\begin{align*}
u_N-u_M=A^{-1}(f(u_N)-f(u_M))+E_{M,N}.
\end{align*}
Multiplying the above equation by $-((f(u_N)-f(u_M))$ and applying the Lipschitz condition \eqref{lip} and the Poincar\'{e} inequality \eqref{poin}, we deduce
\begin{align}\label{exi}
& -\|f\|_{\text{Lip}} \|u_N-u_M\|^2  \\
\le & -(u_N-u_M,f(u_N)-f(u_M)) \nonumber\\
= & -(A^{-1}(f(u_N)-f(u_M)),f(u_N)-f(u_M))-(E_{M,N},f(u_N)-f(u_M)) \nonumber\\
\le & -\gamma\|A^{-1}(f(u_N)-f(u_M))\|^2+\|E_{M,N}\|
\cdot\|f(u_N)-f(u_M)\|. \nonumber 
\end{align}
Using the Young-type inequality
\begin{align*}
\|\phi_1+\phi_2\|^2\geq \epsilon\|\phi_1\|^2-\frac{2-\epsilon}{1-\epsilon}\|\phi_2\|^2,
\quad \forall\ \epsilon\in(0,1),\ \phi_1,\phi_2\in \hh
\end{align*}
with $\phi_1=u_N-u_M,\phi_2=-E_{M,N}$ and $\epsilon=\frac{\|f\|_{\text{Lip}}+\gamma}{2\gamma}$, we obtain
\begin{align*}
&\|A^{-1}(f(u_N)-f(u_M))\|^2 \\
= & \|(u_N-u_M)-E_{M,N}\|^2\nonumber \\
\geq & \frac{\|f\|_{\text{Lip}}+\gamma}{2\gamma}\|u_N-u_M\|^2-\frac{3\gamma-\|f\|_{\text{Lip}}}{\gamma-\|f\|_{\text{Lip}}}\|E_{M,N}\|^2.
\end{align*}
The average inequality $a\cdot b\le \frac{\gamma-\|f\|_{\text{Lip}}}{4\|f\|_{\text{Lip}}^2} a^2+\frac{\|f\|_{\text{Lip}}^2}{\gamma-\|f\|_{\text{Lip}}}b^2$ 
with $a=\|u_N-u_M\|$ and $b=\|E_{M,N}\|$ yields 
\begin{align*}
& \|E_{M,N}\|\cdot\|f(u_N)-f(u_M)\| \\
\le & \frac{\|f\|_{\text{Lip}}^2}{\gamma-\|f\|_{\text{Lip}}}\|E_{M,N}\|^2
+\frac{\gamma-\|f\|_{\text{Lip}}}{4} \|u_N-u_M\|^2. 
\end{align*}
Substituting the above two inequalities into \eqref{exi}, we deduce 
\begin{align}\label{con}
&\|u_N-u_M\|^2
\le 4(3\gamma^2-\|f\|_{\text{Lip}} \gamma+\|f\|_{\text{Lip}}^2)
\|E_{M,N}\|^2.
\end{align}
The moment property of Gaussian process and calculations similarly to \eqref{w-reg} give
\begin{align*}
\ee\left[\|E_{M,N}\|^p\right]
= & C_p\left(\ee\left[\|E_{M,N}\|^2\right]\right)^\frac p2 \\
= & C_p\sum_{k=M+1}^N\sum_{m=1}^\infty 
(A^{-1} Q^\frac 12 \psi_m,\varphi_k )^2,
\end{align*}
which tends to zero as $N,M\rightarrow\infty$ under the conditon \eqref{wq} with $\beta=0$. 
As a consequence, $\{u_N\}$ is a Cauchy sequence in $L^p(\Omega; \dot \hh)$ hence converges to an element $u\in L^p(\Omega; \dot \hh)$. 
The existence of a mild solution of SPDE \eqref{ell} then follows from taking the limit in Eq. \eqref{spe-mild}.

Next we prove the uniqueness. 
Assume that $u,v$ are both solutions of Eq. \eqref{mild}. 
 Arguments similar to the proof of existence yield
\begin{align*}
\|u-v\|
\le 4(3\gamma^2-\|f\|_{\text{Lip}} \gamma+\|f\|_{\text{Lip}}^2)
\|A^{-1}\dot{W}^Q-A^{-1}\dot{W}^Q\|^2=0,
\end{align*}
from which we conclude that $u=v$.

Finally we prove that $u\in L^p(\Omega; \hh^\beta)$ for any $p\ge 1$.
To this end, we first use the Young inequality to obtain 
\begin{align*}
\ee \left[\|u\|_{\beta}^p\right]
&\leq C \ee\left[\|f(u)\|_{\beta-2}^p\right]
+\left(\ee \left[\|\dot{W}^Q\|_{\beta-2}^2\right]\right)^\frac p2. 
\end{align*}
Since the $\hh^\beta$-norm is increasing with respect to 
$\beta\in [0,2]$, \eqref{lip} implies that 
\begin{align*}
\ee \left[\|f(u)\|_{\beta-2}^p\right]
\le \ee\left[\|f(u)\|^p\right]
\le C(1+\ee\left[\|u\|^p\right]) <\infty.
\end{align*}
Substituting \eqref{w-reg} into the above two inequalities, we conclude that 
$\ee [\|u\|_{\beta}^2]<\infty$. The proof is complete. $\Box$\\

\subsection{Error Estimates for Spectral Truncations}

In this subsection we estimate the error between the solution of \eqref{eq:app_SPDE} and the exact solution of SPDE \eqref{ell}. 
First we consider the regularity spectral projection of the noise given by $\pp_N\dot{W}^Q$. Note that 
 $\{\lambda_m\}_{m=1}^\infty$ is increasing, thus for any $\alpha\geq \beta-2$ and any $N\in \nn_+$, we have 
\begin{align}\label{spe-w-reg}
\ee \left[\|\pp_N\dot{W}^Q\|_\alpha^2\right]
&=\ee \left[\left\|\sum_{k=1}^\infty A^{\frac {\alpha}2}\pp_N Q^{\frac 12}\psi_k \eta_k \right\|^2\right] \nonumber \\
&=\ee \left[\sum_{m=1}^\infty \left(\sum_{k=1}^\infty 
\left(\pp_NA^{\frac {\alpha}2} Q^{\frac 12} \psi_k,\varphi_m \right) \eta_k\right)^2 \right] \nonumber  \\
&=\sum_{m=1}^N \sum_{k=1}^\infty \lambda_m^{2-\beta+\alpha}
\left(A^{\frac {\beta-2}2} Q^{\frac 12}\psi_k,\varphi_m \right)^2  \nonumber  \\
& \le \lambda_N^{2-\beta+\alpha} \|A^{\frac {\beta-2}2}\|_{\LL_2^0}^2.
\end{align}

Let $E_N:=A^{-1}(I-\pp_N)\dot{W}^Q$.
Then for any $\alpha\in [0,\beta]$ and any $N\in \nn_+$, 
\begin{align}\label{spe-w-err}
\ee\left[\|E_N\|_\alpha^2\right] 
& =\sum_{m=N+1}^\infty \sum_{k=1}^\infty \lambda_m^{\alpha-\beta}
\left(A^{\frac {\beta-2}2} Q^{\frac 12}\psi_k,\varphi_m \right)^2 \nonumber  \\ 
& \le \lambda_{N+1}^{\alpha-\beta} \|A^{\frac {\beta-2}2}\|_{\LL_2^0}^2.
\end{align}

The above calculations lead to the following error estimation between the solution $u_N$ of Eq. \eqref{spe-mild} and the solution $u$ of Eq. \eqref{mild},
as well as the Sobolev regularity of $u_N$, which will be used in the error estimation of finite element approximations.

\begin{tm}\label{well-un}
Let $p\ge 1$.  Assume that \ref{f} and \ref{w} hold and $u$ and $u_N$, $N\in \nn_+$, are the solutions of Eq. \eqref{mild} and Eq. \eqref{spe-mild}, respectively. 
Then $u_N\in L^p(\Omega; \dot\hh^2)$ and there exists a constant $C$ such that 
\begin{align}\label{spe-reg1}
\ee\left[\|u_N\|_2^p\right]
\le C\lambda_N^\frac{(2-\beta)p}2 
(1+\|A^{\frac {\beta-2}2}\|_{\LL_2^0}^p).
\end{align}
Assume furthermore that $f$ has bounded derivatives up to order $r-1$ with 
$r\geq 2$. Then $u_N\in L^p(\Omega; \dot\hh^{r+1})$ and 
\begin{align}\label{spe-reg2}
\ee\left[\|u_N\|_{r+1}^p\right]
\le C\lambda_N^{\frac{(r+1-\beta)p}2} 
(1+\|A^{\frac {\beta-2}2}\|_{\LL_2^0}^p).
\end{align}
Moreover,
\begin{align}\label{uun}
\left(\ee\left[\|u-u_N\|^p\right]\right)^\frac 1p
\le C\lambda_{N+1}^{-\frac {\beta} 2} 
\left( 1+\|A^{\frac {\beta-2}2}\|_{\LL_2^0} \right).
\end{align}
\end{tm}

Proof: 
We first prove \eqref{spe-reg1} and \eqref{spe-reg2}.
Since $f$ is Lipschitz continuous, there exists a constant $C$ such that
\begin{align*}
\|u_N\|_2=\|f(u_N)+\pp_N \dot{W}^Q\|
\le C(1+\|u_N\|+\|\pp_N \dot{W}^Q\|).
\end{align*}
Taking inner product with $u_N$ in Eq. \eqref{spe-mild}, using integration by parts formula and Poincar\'{e} inequality \eqref{poin}, we have that 
\begin{align*}
&(\gamma-\|f\|_{\text{Lip}})\|u_N\|^2-|f(0)|\cdot \|u_N\| 
\\
\le & (\nabla u_N,\nabla u_N)-(f(u_N),u_N)\\
= & (\pp_N \dot{W}^Q,u_N) 
\le \|\pp_N \dot{W}^Q\|\cdot \|u_N\|,
\end{align*}
from which we obtain
\begin{align}\label{uw}
\|u_N\|\le \frac{|f(0)|+\|\pp_N \dot{W}^Q\|}{\gamma-\|f\|_{\text{Lip}}}.
\end{align}
We conclude \eqref{spe-reg1} by combining the above equations and \eqref{spe-w-reg} with $\alpha=0$.
By recursion, we obtain \eqref{spe-reg2} by \eqref{spe-w-reg} with $\alpha=r-1$.

To prove \eqref{uun}, we subtract Eq. \eqref{spe-mild} from Eq. \eqref{mild} to obtain 
\begin{align*}
u-u_N=A^{-1}(f(u)-f(u_N))+E_N. 
\end{align*}
Similarly to \eqref{con}, we have
 \begin{align}\label{con1}
&\|u-u_N\|^2
\le 4(3\gamma^2-\|f\|_{\text{Lip}} \gamma+\|f\|_{\text{Lip}}^2) 
\|E_N\|^2.
\end{align}
Substituting the estimations \eqref{uw} and \eqref{spe-w-err} for $E_N$, we obtain \eqref{uun}.
$\Box$\\
\begin{rk}
We remark that the well-posedness of SPDE \eqref{ell} is also valid for non-Lipschitz assumptions on $f$ possibly depending on the spatial variable, which was proposed in \cite{CHL18, CYY07}. In particular, we may assume that there exist positive constants $L_1<\gamma$ and $L_2$ such that for any $x\in \OOO$ and any $u,v\in \rr$,
\begin{align*}
(f(x,u)-f(x,v),u-v)\geq -L_1|u-v|^2
\end{align*}
and
\begin{align*}
|f(x,u)-f(x,v)|\le L_2(1+|u-v|).
\end{align*}
Similar to \cite{CHL18, CYY07}, our arguments for spectral projection approximations and finite element approximations are also valid under the above assumption on $f$.
In that case, all the convergence rates halve.
\end{rk}

\section{Finite Element Approximations}
\label{sec-fem}

In this section, we establish a general framework of constructing finite element approximations of the spectrally truncated noise driven Eq. \eqref{spe-mild} and derive their error estimates.
Then we apply this framework to the discretization of SPDE \eqref{ell} driven by power-law noises.

\subsection{Finite Element Approximations}
\label{subsec-fem}

Let ${\mathcal{T}_h}$ be a quasiuniform family of triangulations of $\OOO$ with meshsize $h\in(0,1)$. We choose the finite element space $V_h$ consisting of continuous piecewise polynomials of degree $r$ such that
\begin{align}
\inf_{v\in V_h} \|v-v_h\|_{\hh^\bs}
\leq C h^{k-\bs}\|v\|_{\hh^k},
\quad \forall\ v\in \hh^k,\ \bs\le k\le r+1. \label{int-fem}
\end{align}
The variational formulation of Eq. \eqref{spe-mild} is to find 
$u_N\in \hh_0^1$ such that
\begin{align}\label{spe-var}
(\nabla u_N,\nabla v)=(f(u_N),v)+(\pp_N\dot{W}^Q,v),\quad \forall\ v\in \hh_0^1. 
\end{align}
Then the finite element approximation to \eqref{spe-var} is to find $u_N^h\in V_h$ such that
\begin{align}\label{unh}
(\nabla u_N^h,\nabla v)=(f(u_N^h),v)+(\pp_N \dot{W}^Q,v),\quad \forall\ v\in V_h. 
\end{align}

In order to estimate the error $u_N-u_N^h$, we need the Ritz projection operator $\RR_h: \hh_0^1( \OOO)\rightarrow V_h$ defined by
\begin{align}\label{pro}
(\nabla \RR_h w,\nabla v)=(\nabla w,\nabla v),\ \forall\ v\in V_h,\ w\in \hh_0^1( \OOO). 
\end{align}
It is well-known that (see e.g. \cite{Tho06})
\begin{align}\label{pro-err}
\|w-\RR_h w\|
\leq C h^{r+1}\|w\|_{\hh^{r+1}},\ \forall\ w\in \hh_0^1 \cap \hh^{r+1}. 
\end{align}

\begin{tm}\label{ununh}
Let $p\ge 1$ and Assumptions \ref{f} and \ref{w} hold, and $u_N$ and $u_N^h$ be the solutions of Eq. \eqref{spe-mild} and Eq. \eqref{unh}, respectively. 
Then there exists a constant $C$ such that 
\begin{align}\label{ununh1}
\left(\ee\left[\|u_N-u_N^h\|^p\right]\right)^\frac1p
\le Ch^2\lambda_N^{\frac{2-\beta}2} 
\left( 1+\|A^{\frac {\beta-2}2}\|_{\LL_2^0}\right).
\end{align}
Assume furthermore that $f$ has bounded derivatives up to order $r-1$ for some $r\geq 2$. Then
\begin{align}\label{ununh2}
\left(\ee\left[\|u_N-u_N^h\|^p\right]\right)^\frac1p
\le Ch^{r+1} \lambda_N^{\frac{r+1-\beta}2} 
\left( 1+\|A^{\frac {\beta-2}2}\|_{\LL_2^0} \right).
\end{align}
\end{tm}

Proof: 
It follows from \eqref{spe-var}, Eq. \eqref{unh} and \eqref{pro} that
\begin{align}
& (\nabla(\RR_hu_N-u_N^h),\nabla(\RR_hu_N-u_N^h)) \nonumber \\
= & (f(u_N)-f(u_N^h),\RR_hu_N-u_N^h).
\end{align}
 Eq. \eqref{lip} and the average inequality 
 \begin{align*}
a\cdot b
\le \frac{\gamma-\|f\|_{\text{Lip}}}{2\|f\|_{\text{Lip}}^2} a^2+\frac{\|f\|_{\text{Lip}}^2}{2(\gamma-\|f\|_{\text{Lip}})}b^2
\end{align*} 
with $a=\|u_N-u_N^h\|$ and $b=\|u_N-\RR_hu_N\|$ yield
\begin{align}
&\|\nabla(\RR_hu_N-u_N^h)\|^2 \nonumber\\
= & (f(u_N)-f(u_N^h),\RR_hu_N-u_N)
+(f(u_N)-f(u_N^h),u_N-u_N^h) \nonumber\\
\le & \frac{\gamma+\|f\|_{\text{Lip}}}{2}\|u_N-u_N^h\|^2+\frac{\|f\|_{\text{Lip}}^2}{2(\gamma-\|f\|_{\text{Lip}})}\|\RR_hu_N-u_N\|^2.
\end{align}

Applying the triangle inequality, the Poincar\'{e} inequality (2.4), and the above inequality, we obtain
\begin{align}
& \gamma \|u_N-u_N^h\|^2 \nonumber \\
\le & \gamma \|u_N-\RR_h u_N\|^2 + \gamma \|\RR_h u_N-u_N^h\|^2 \nonumber \\
\le & \gamma \|u_N-\RR_h u_N)\|^2 + \|\nabla(\RR_h u_N-u_N^h)\|^2 \nonumber\\
\le & \frac{\gamma+\|f\|_{\text{Lip}}}{2}\|u_N-u_N^h\|^2
+ \Big(\gamma + \frac{\|f\|_{\text{Lip}}^2}{2(\gamma-\|f\|_{\text{Lip}})} \Big) 
\|\RR_hu_N-u_N\|^2. 
\end{align}
From the above and the standard estimation \eqref{pro-err} with $r=1$, we have 
\begin{align}\label{unpun}
\|u_N-u_N^h\|
\le C\|u_N-\RR_hu_N\|
\le Ch^2\|u_N\|_2,
\end{align}
and thus \eqref{ununh1} holds.
By \eqref{pro-err} and \eqref{spe-reg2} in Theorem \ref{well-un}, we obtain \eqref{ununh2}.
$\Box$\\

The next theorem states the main result of the paper. 
\begin{tm} \label{uunh}
Let $p\ge 1$ and Assumptions \ref{f} and \ref{w} hold, and $u$ and $u_N^h$ be the solutions of SPDE \eqref{ell} and Eq. \eqref{unh}, respectively. 
Then there exists a constant $C$ such that 
\begin{align}\label{uunh1}
\left(\ee\left[\|u-u_N^h\|^p\right]\right)^\frac1p
\le C\left( N^{-\frac {\beta} d}+h^2 N^{\frac{2-\beta} d} \right) 
\left( 1+\|A^{\frac {\beta-2}2}\|_{\LL_2^0} \right).
\end{align}
In particular, if $f$ has bounded derivatives up to order $r-1$ for some $r\geq 2$, then 
\begin{align}\label{uunh2}
\left(\ee\left[\|u-u_N^h\|^p\right]\right)^\frac1p
\le C\left( N^{-\frac {\beta} d}
+h^{r+1} N^{\frac{r+1-\beta} d} \right)
\left( 1+\|A^{\frac {\beta-2}2}\|_{\LL_2^0} \right).
\end{align}
\end{tm}

\textbf{Proof:}  
The estimations \eqref{uunh1}--\eqref{uunh2} follows immediately from \eqref{uun} in Theorem \ref{well-un} and \eqref{ununh1}--\eqref{ununh2} in Theorem \ref{ununh} as well as Weyl's law \eqref{weyl}. $\Box$\\

\begin{rk}
When $h=\OO (N^{-\frac 1d})$, we obtain the optimal convergence rate, independent of the choice of $r$:
\begin{align}
\left(\ee\left[\|u-u_N^h\|^p\right]\right)^\frac1p
\le Ch^\beta \left( 1+\|A^{\frac {\beta-2}2}\|_{\LL_2^0} \right),
\end{align}
which coincides with the regularity established in Theorem \ref{well}.
\end{rk}

\subsection{Applications to Power-law Noises}
\label{subsec-pow}

In this subsection we apply the error estimate results to SPDEs driven by the power-law noises, where the covariance operator $Q$ is given by $Q=A^\rho$, $\rho\in \rr$.
\begin{tm}
Let $p\ge 1$ and Assumption \ref{f} hold. Then we have the following well-posedness and error estimate results. 

(i) The power-law noise driven SPDE \eqref{ell} has a unique mild solution (see \eqref{mild}) if and only if 
$\rho<2-\frac d2$. Moreover, in case the mild solution $u$ exists, $u\in \hh^{2-\frac d2-\rho-\epsilon}$ a.s. for any positive $\epsilon$

(ii) If  $-\frac d2<\rho<2-\frac d2$, then
\begin{align}\label{pow-1}
\left(\ee\left[\|u-u_N^h\|^p\right]\right)^\frac1p
\le C(N^{\frac {\rho-2}d+\frac 12}
+h^2 N^{\frac{\rho}d+\frac 12}). 
\end{align}
If, in addition, $f$ has bounded derivatives up to order $r-1$ for some 
$r\geq 2$, then
\begin{align}\label{pow-2}
\left(\ee\left[\|u-u_N^h\|^p\right]\right)^\frac1p
\le C(N^{\frac {\rho-2}d+\frac 12}
+h^{r+1}N^{\frac{\rho+r-1}d +\frac 12}).
\end{align}
\end{tm}

\textbf{Proof:}  
It suffices to verify that the conditions of Theorem \ref{well}, Theorem \ref{well-un}, and Theorem \ref{uunh} hold.
Set $Q=A^\rho$, we get, by Weyl's law \eqref{weyl},
\begin{align}
\|A^{\frac {\beta-2} 2}\|_{\LL_2^0}^2
=\|A^{\frac {\beta-2+\rho} 2}\|_{HS}^2=
\sum_{k=1}^\infty \lambda_k^{\beta-2+\rho}
\asymp \sum_{k=1}^\infty k^{\frac {2(\beta-2+\rho)}d}. 
\end{align}
The above series converges if and only if $\beta<2-\frac d2-\rho$, which is the condition \eqref{wq} in Theorem \ref{well}. 
(i) then follows from Theorem \ref{well}.

Applying Weyl's law \eqref{weyl}, we deduce from \eqref{con1} in Theorem \ref{well-un} and \eqref{spe-w-err} with $\alpha=0$ that
\begin{align*}
\ee\left[\|u-u_N^h\|^p\right]
& \le C \ee\left[\|E_N\|^p\right] \\
& = C\left( \sum_{k=N+1}^\infty \lambda_k^{\rho-2} \right)^\frac p2
\asymp N^{(\frac {\rho-2}d+\frac 12)p}.
\end{align*}
Analogously, by \eqref{unpun} in Theorem \ref{ununh} and \eqref{spe-w-reg} with $\alpha=0$, 
\begin{align*}
\ee\left[\|u_N-u_N^h\|^p\right]
& \leq C h^{2p} \ee\left[\|u_N\|_2^p\right] \\
& \leq C h^{2p} \left( \sum_{k=1}^N \lambda_k^\rho \right)^\frac p2
\asymp h^{2p} \left( \sum_{k=1}^N k^\frac{2\rho}{d} \right)^\frac p2.
\end{align*}
 \eqref{pow-1} then follows immediately from the elementary inequality
\begin{align*}
\sum_{k=1}^N k^\delta\asymp N^{\delta+1},\quad \delta>-1.
\end{align*}
The estimation \eqref{pow-2} follows from a similar argument and 
\eqref{spe-w-reg} with $\alpha=r-1$.

\begin{rk}\label{main}
Let $h=N^{-\frac 1d}$. Then we have the optimal error estimate
\begin{align*}
\left(\ee\left[\|u-u_N^h\|^p\right]\right)^\frac1p
\le Ch^{2-\frac{d}{2}-\rho}.
\end{align*}
In particular for white noise driven SPDE \eqref{ell}, i.e., $\rho=0$, we have
\begin{align}\label{eq:1d} 
\left(\ee\left[\|u-u_N^h\|^p\right]\right)^\frac1p
\le Ch^{2-\frac{d}{2}}. 
\end{align}
We can see from the above estimation that for $d=2$, the convergence of our finite element approximation for the white noise driven SPDE \eqref{ell} is first order. This is in comparison with the error estimate 
\begin{align*}
\left(\ee\left[\|u-u_h\|^2\right]\right)^\frac12
\le Ch^{1-\epsilon},
\end{align*}
for the finite element approximation of the same problem with noise term approximated by piecewise constants in space (see e.g. \cite{CYY07, GM06}). Here $\epsilon$ is a small and positive number. Thus our error estimate is an improvement. 
Moreover, for $d=1$, the convergence order of our method is $1.5$,
which improves the first order convergence results in \cite{ANZ98, DZ02, GM06}. 
\begin{rk}
When $\OO=\Pi_{i=1}^d[a_i,b_i]$, our result for the white noise case is comparable to the one in \cite{Zhang2016}. 
\end{rk}

\end{rk}

\section*{Acknowledgements}

We thank the anonymous referee for very helpful remarks and suggestions.
The first author is partially supported by U.S. National Science Foundation, No. DMS1620150, and U.S. Army ARDEC, No. W911SR-14-2-0001. 
The second author is partially supported by National Natural Science Foundation of China, No. 91130003, No. 11021101, and No. 11290142.
The last is partially supported by Hong Kong RGC General Research Fund, No. 16307319, and the UGC - Research Infrastructure Grant, No. IRS20SC39. 

\bibliographystyle{plain}
\bibliography{bib}

\end{document}